\documentclass[11pt]{article}
\renewcommand{\baselinestretch}{1.2}
\newcommand{\dated}{\mbox{} \hfill {\small [{\tt \today}]}} \usepackage{amsmath,amssymb,amsfonts,diagrams}
\newenvironment{keywords}{\noindent\small {\it Keywords\/}:}{\vskip 4pt}
\newenvironment{classification}{\noindent\small 2000 {\it Mathematics Subject
Classification\/}:}{\vskip 12pt}

\newcommand{\comps}{{\mathbb C}}

\newcommand{\ints}{{\mathbb Z}}

\newcommand{\free}{{\mathbb F}}

\newcommand{\tensor}{\otimes}

\newcommand{\cstar}{{C^\ast}}
\newcommand{\wstar}{{W^\ast}}

\newcommand{\id}{{\mathrm{id}}}

\newcommand{\A}{{\mathfrak A}}

\usepackage{amsthm,enumerate}
\theoremstyle{plain}
\newtheorem{theorem}{Theorem}[section]
\newtheorem{lemma}[theorem]{Lemma}
\newtheorem{corollary}[theorem]{Corollary}
\newtheorem{proposition}[theorem]{Proposition}
\theoremstyle{definition}
\newtheorem{definition}[theorem]{Definition}
\theoremstyle{remark}
\newtheorem*{remark}{Remark}
\newtheorem*{example}{Example}
\newtheorem*{rems}{Remarks}
\newtheorem*{exs}{Examples}
\newenvironment{remarks}{\begin{rems}\begin{enumerate}}{\end{enumerate}\end{rems}}
\newenvironment{examples}{\begin{exs}\begin{enumerate}}{\end{enumerate}\end{exs}}
\newenvironment{items}{\begin{enumerate}[\rm (i)]}{\end{enumerate}}
\newenvironment{alphitems}{\begin{enumerate}[\rm (a)]}{\end{enumerate}}

\newcommand{\Csw}{\mathcal{C}_\sigma^w}
\newcommand{\LCsw}{\mathcal{LC}_\sigma^w}
\newcommand{\RCsw}{\mathcal{RC}_\sigma^w}
\title{Connes-amenability of Fourier--Stieltjes algebras}
\author{\textit{Volker Runde}\thanks{Research supported by NSERC.} \and \textit{Faruk Uygul}}
\date{}
\begin{document}
\maketitle
\begin{abstract}
Let $G$ be a locally compact group, and let $B(G)$ denote its Fourier--Stieltjes algebra. We show that $B(G)$ is Connes-amenable if and only if $G$ is almost abelian.
\end{abstract}
\section*{Introduction}
Loosely speaking, a \emph{dual Banach algebra} is a Banach algebra that is also a dual Banach space such that multiplication is separately weak$^\ast$ continuous. Examples of dual Banach algebras are, for instance, von Neumann algebras and measure algebras of locally compact groups. The name tag ``dual Banach algebra'' was introduced in \cite{Runde01}, but the concept is much older. Already in \cite{JohnsonKadisonRingrose72}, B.\ E.\ Johnson, R.\ V.\ Kadison, and J.\ R.\ Ringrose used them in their development of normal cohomology of von Neumann algebras: they did some of that development in the framework of general dual Banach algebras, noting that this was all they needed from a von Neumann algebra. Indeed, a surprisingly large part of von Neumann algebra theory can be extended to general dual Banach algebras: for instance, they have a rich representation that parallels that of von Neumann algebras, in which reflexive Banach spaces play the r\^ole of Hilbert spaces (\cite{Daws07}), and in which a version of von Neumann's bicommutant theory holds true (\cite{Daws11}).
\par 
In his memoir \cite{Johnson72a}, B.\ E.\ Johnson initiated the theory of amenable Banach algebras, which has been a very active branch of mathematics ever since. It seems, however, that amenability in the sense of \cite{Johnson72a} is not very well suited for the study of dual Banach algebras: it appears to be too strong to allow for many interesting examples. For instance, von Neumann algebras are amenable if and only if they are subhomogeneous (\cite{Wassermann76}), and the measure algebra $M(G)$ of a locally compact group $G$ is amenable if and only if $G$ is discrete and amenable (\cite{DalesGhahramaniHelemskii02}). 
\par  
In his groundbreaking paper \cite{Connes76}, A.\ Connes introduced a notion of amenability for von Neumann algebras that takes the dual space structure into account (it was later labeled \emph{Connes-amenability} in his honor by A.\ Ya.\ Helemski\u{\i} in \cite{Helemskii91}). Connes-amenability is far better suited for the study of von Neumann algebras as it is equivalent to important von Neumann algebraic properties, such as injectivity or hyperfiniteness. Also, for measure algebras it is far better suited: the measure algebra $M(G)$ of a locally compact group $G$ is Connes-amenable if and only if $G$ is amenable (\cite{Runde03a}). 
\par
The \emph{Fourier--Stieltjes algebra} $B(G)$ of a locally compact group $G$ was introduced by P.\ Eymard in \cite{Eymard64}. It is a dual Banach algebra with natural predual $\cstar(G)$, the (full) group $\cstar$-algebra of $G$. The natural question arises which locally compact groups $G$ have a Connes-amenable Fourier--Stieltjes algebra $B(G)$. With an eye on the main results of \cite{Runde03a} and \cite{ForrestRunde05}, one is led to the conjecture that $B(G)$ is Connes-amenable if and only if $G$ is \emph{almost abelian}, i.e., has an abelian subgroup of finite index (with the ``if'' part being fairly straightforward; see \cite[Proposition 3.1]{Runde04}). There has been some circumstantial evidence suggesting that the ``only if'' part of the conjecture might also be true, meaning that it has been corroborated in certain special cases: it definitely holds true if $G$ is discrete and amenable (\cite[Theorem 3.5]{Runde04}) or the topological product of a family of finite groups (\cite[Theorem 3.4]{Runde04}).
\par 
In the present paper, we prove this conjecture in full generality. 
\section{Dual Banach algebras and Connes-amenability}
We begin with a formally rigorous definition of a dual Banach algebra:
\begin{definition} \label{dualdef}
A \emph{dual Banach algebra} is a pair $(\A,\A_\ast)$ of Banach spaces such that: 
\begin{alphitems}
\item $\A = (\A_\ast)^\ast$; 
\item $\A$ is a Banach algebra, and multiplication in $\A$ is separately $\sigma(\A,\A_\ast)$ continuous.
\end{alphitems}
\end{definition}
\begin{examples}
\item Every von Neumann algebra is a dual Banach algebra; the predual space $\A_\ast$ is then unique (\cite[Corollary III.3.9]{Takesaki02}).
\item Let $G$ be a locally compact group, let $\mathcal{C}_0(G)$ denote the continuous $\comps$-valued functions on $G$ vanishing at $\infty$, and let $M(G)$ denote the measure algebra of $G$. Then $(M(G),\mathcal{C}_0(G))$ is a dual Banach algebra.
\item Let $G$ be a locally compact group, let $\cstar(G)$ and $C^\ast_r(G)$ denote the full and the reduced group $\cstar$-algebra of $G$, respectively, and let $B(G)$ and $B_r(G)$ be the Fourier--Stieltjes algebra and the reduced Fourier--Stieltjes algebra of $G$, respectively (as defined in \cite{Eymard64}). Then both $(B(G),\cstar(G))$ and $(B_r(G),C^\ast_r(G))$ are dual Banach algebras. 
\end{examples}
\begin{remarks}
\item The seemingly pedantic formulation of Definition \ref{dualdef} is necessary because the predual space $\A_\ast$ need not be unique: in \cite{DawsHaydonSchlumprechtWhite12}, the authors construct a continuum of different preduals of the convolution algebra $\ell^1(\ints)$, each of which is isometrically isomorphic to $c_0(\ints)$, and each of which turns $\ell^1(\ints)$ into a dual Banach algebra.
\item In many cases, as in the examples above, there is a canonical predual; in such cases, we shall not explicitly refer to it and simply call $\A$ a dual Banach algebra.
\end{remarks}
\begin{definition}
Let $(\A,\A_\ast)$ be a dual Banach algebra, and let $E$ be a right Banach $\A$-module. Then we call the left Banach $\A$-bimodule $E^\ast$ \emph{normal} if, for each $\phi \in E^\ast$, the maps
\[
  \A \to E^\ast, \quad a \mapsto a \cdot \phi
\]
is $\sigma(\A,\A_\ast)$-$\sigma(E^\ast,E)$ continuous.
\end{definition}
\par 
Analogously, one defines, for a dual Banach algebra $\A$, normal right Banach $\A$-bimodules and normal Banach $\A$-bimodules.
\par 
The notion of a normal Banach bimodule over a dual Banach algebra is crucial for the definition of Connes-amenability:
\begin{definition} \label{Connesdef}
A dual Banach algebra $(\A,\A_\ast)$ is \emph{Connes amenable} if, for every Banach $\A$-bimodule $E$ such that $E^\ast$ is normal, every $\sigma(\A,\A_\ast)$-$\sigma(E^\ast,E)$ continuous derivation $D \!: \A \to E^\ast$ is inner.
\end{definition}
\begin{remark}
In \cite[Remark 5.33]{Connes76}, A.\ Connes referred to von Neumann algebras satisfying Definition \ref{Connesdef} as ``amenable''. In order to tell them apart from amenable von Neumann algebras in the sense of \cite{Johnson72a}, A.\ Ya.\ Helemski\u{\i} introduced the moniker ``Connes-amenable'' in \cite{Helemskii91}.
\end{remark}
\begin{examples}
\item For von Neumann algebras, Connes-amenability is equivalent to a number of important von Neumann algebraic properties, such as injectivity, semidiscreteness, and hyperfiniteness: see \cite{Takesaki03} and also \cite[Chapter 6]{Runde02}. On the other hand, a von Neumann algebra is amenable in the sense of \cite{Johnson72a} if and only if it is subhomogeneous (\cite{Wassermann76}).
\item Let $G$ be a locally compact group. Then $M(G)$ is Connes-amenable if and only if $G$ is amenable (\cite{Runde03a}) whereas $M(G)$ is amenable in the sense of \cite{Johnson72a} if and only if $G$ is \emph{discrete} and amenable (\cite{DalesGhahramaniHelemskii02}).
\end{examples}
\par 
These examples suggest that Banach algebraic amenability, as introduced in \cite{Johnson72a}, is too strong a condition to impose to dual Banach algebras and that Connes-amenability is far better suited for the study of dual Banach algebras.
\section{Connes-amenability and diagonal type elements}
The following definition is from \cite{Runde04}:
\begin{definition}
Let $(\A,\A_\ast)$ be a dual Banach algebra, and let $E$ be a left Banach $\A$-module. An element $x \in E$ is called \emph{left $\sigma$-weakly continuous} if the map
\[
  \A \to E, \quad a \mapsto a \cdot x
\]
is $\sigma(\A,\A_\ast)$-weakly continuous. We denote the collection of all left $\sigma$-weakly continuous elements of $E$ by $\LCsw(E)$. 
\end{definition}
\begin{remarks}
\item If $E$ is a right Banach $\A$-module, we define the set $\RCsw(E)$ of all \emph{right $\sigma$-weakly continuous} elements of $E$ in the obvious way, and if $E$ is a Banach $\A$-bimodule, we set $\Csw(E) := \LCsw(E) \cap \RCsw(E)$ and simply speak of the \emph{$\sigma$-weakly continuous} elements of $E$.
\item It is immediate that $\A_\ast \subset \Csw(\A^\ast)$ for every dual Banach algebra $(\A,\A_\ast)$; if $\A$ has an identity, then $\A_\ast = \LCsw(\A^\ast) = \RCsw(\A^\ast)$ holds.
\item Let $E$ be a non-reflexive Banach space, and turn $E^\ast$ into a Banach algebra by letting $\phi \psi := 0$ for $\phi,\psi \in E^\ast$. Then $(E^\ast,E)$ is trivially a dual Banach algebra without identity, and we have $E \subsetneq E^{\ast\ast} = \Csw(E^{\ast\ast})$.
\end{remarks}
\par
The verification of the following is routine (compare \cite{Runde04}):
\begin{proposition} \label{LCswprop}
Let $(\A,\A_\ast)$ be a dual Banach algebra, and let $E$ be a left Banach $\A$-module. Then:
\begin{items}
\item $\LCsw(E)$ is a closed submodule of $E$;
\item if $F$ is another left Banach $\A$-bimodule and $\theta \!: E \to F$ is a bounded homomorphism of left $\A$-modules, then $\theta(\LCsw(E)) \subset \LCsw(F)$ holds;
\item $\LCsw(E)^\ast$ is a normal right Banach $\A$-module;
\item $E^\ast$ is a normal right Banach $\A$-module if and only if $E = \LCsw(E)$.
\end{items}
\end{proposition}
\begin{remark}
Of course, statements analogous to those of Proposition \ref{LCswprop} hold for right modules and bimodules.
\end{remark} 
\par
Following \cite{EffrosRuan00}, we denote the (completed) projective tensor product of Banach spaces by $\tensor^\gamma$. If $\A$ is a Banach algebra, $E$ is a left Banach $\A$-module, and $F$ is a right Banach $\A$-module, then $E \tensor^\gamma F$ is a Banach $\A$-bimodule via
\[
  a \cdot (x \tensor y) := a \cdot x \tensor y \quad\text{and}\quad (x \tensor y) \cdot a := x \tensor y \cdot a \qquad (a \in \A, \, x \in E, \, y \in F).
\]
In particular, $\A \tensor^\gamma \A$ is a Banach $\A$-bimodule in a canonical manner such that the multiplication map
\[
  \Delta \!: \A \tensor^\gamma \A \to \A, \quad a \tensor b \mapsto ab
\]
is a bimodule homomorphism. A \emph{virtual diagonal} for $\A$ is an element $\boldsymbol{D} \in (\A \tensor^\gamma \A)^{\ast\ast}$ such that
\[
  a \cdot \boldsymbol{D} = \boldsymbol{D} \cdot a \quad\text{and}\quad a \Delta^{\ast\ast}\boldsymbol{D} = a \qquad (a \in \A)
\]
In \cite{Johnson72b}, B.\ E.\ Johnson showed that a Banach algebra is amenable if and only if it has a virtual diagonal.
\par 
There is an analogous characterization of Connes-amenable dual Banach algebras in terms of suitable diagonal type elements.
\par 
Let $(\A,\A_\ast)$ be a dual Banach algebra. As $\A_\ast \subset \Csw(\A^\ast)$, it follows from (the bimodule analog of) Proposition \ref{LCswprop}(ii) that $\Delta^\ast \A_\ast \subset \Csw((\A \tensor^\gamma \A)^\ast)$, so that $\Delta^{\ast\ast} \!: (\A \tensor^\gamma \A)^{\ast\ast} \to \A^{\ast\ast}$ induces a Banach $\A$-bimodule homomorphism $\Delta_\sigma^w \!: \Csw((\A \tensor^\gamma \A)^\ast)^\ast \to \A$.
\par
We define (see \cite{Runde04}):
\begin{definition} \label{CswDdef} 
Let $(\A,\A_\ast)$ be a dual Banach algebra. Then $\boldsymbol{D} \in \Csw((\A \tensor^\gamma \A)^\ast)^\ast$ is called a \emph{$\Csw$-virtual diagonal} for $(\A,\A_\ast)$ if
\[
  a \cdot \boldsymbol{D} = \boldsymbol{D} \cdot a \quad\text{and}\quad a \Delta_\sigma^w \boldsymbol{D} = a \qquad (a \in \A).
\]
\end{definition}
\par 
As it turns out, the existence of a $\Csw$-diagonal characterizes the Connes-amenable dual Banach algebras (\cite[Theorem 4.8]{Runde04}):
\begin{theorem} \label{CswDthm}
The following are equivalent for a dual Banach algebra $(\A,\A_\ast)$:
\begin{items}
\item $(\A,\A_\ast)$ is Connes-amenable;
\item there is a $\Csw$-virtual diagonal for $(\A,\A_\ast)$.
\end{items} 
\end{theorem}
\par 
We now look at another class of diagonal type elements for dual Banach algebras. Given a dual Banach algebra $(\A,\A_\ast)$, we denote by $\mathcal{B}^2_\sigma(\A,\comps)$ the separately $\sigma(\A,\A_\ast)$ continuous bilinear maps from $\A \times \A$ to $\comps$; it can be identified with a closed submodule of $(\A \tensor^\gamma \A)^\ast$. It is clear that $\Delta^\ast \A_\ast \subset \mathcal{B}_\sigma^2(\A,\comps)$, so that $\Delta^{\ast\ast}$ induces a Banach $\A$-bimodule homomorphism $\Delta_\sigma \!: \mathcal{B}^2_\sigma(\A,\comps)^\ast \to \A$.
\par
The following definition is from \cite{Effros88}, where it was formulated for von Neumann algebras:
\begin{definition} 
Let $(\A,\A_\ast)$ be a dual Banach algebra. Then $\boldsymbol{D} \in \mathcal{B}^2_\sigma(\A, \comps)^\ast$ is called a \emph{normal, virtual diagonal} for $(\A,\A_\ast)$ if
\[
  a \cdot \boldsymbol{D} = \boldsymbol{D} \cdot a \quad\text{and}\quad a \Delta_\sigma \boldsymbol{D} = a \qquad (a \in \A).
\]
\end{definition}
\par 
The corollary below is well known: it was first already proven for von Neumann algebras in \cite{Effros88}, and it was observed in \cite{CorachGale98} that the proof carries over to general dual Banach algebras. Still, we give an alternative proof that invokes Theorem \ref{CswDthm} instead of the definition of Connes-amenability.
\begin{corollary} \label{VDcor} 
Let $(\A,\A_\ast)$ be a dual Banach with a normal, virtual diagonal. Then $(\A,\A_\ast)$ is Connes-amenable.
\end{corollary}
\begin{proof}
We claim that $\Csw((\A \tensor^\gamma \A)^\ast) \subset \mathcal{B}^2_\sigma(\A, \comps)$ (with the canonical identifications in place, of course).
\par 
Fix $a \in \A$, and define
\[
  \rho_a \!: \A \to \A \tensor \A, \quad b \mapsto a \tensor b.
\]
Then $\rho_a$ is a homomorphism of right Banach $\A$-modules, so that $\rho_a^\ast$ is a homomorphism of left Banach $\A$-modules and, consequently, $\rho_a^\ast(\LCsw((\A \tensor^\gamma \A)^\ast) \subset \LCsw(\A^\ast) = \A_\ast$ by Proposition \ref{LCswprop}(ii). In view of the definition of $\rho_a$ and the fact that $a$ was arbitrary, this means that every element of $\Csw((\A \tensor^\gamma \A)^\ast)$ is weak$^\ast$ continuous in the second variable. Analogously, one sees that the same is true for the first variable.
\par 
Hence, if $\boldsymbol{D} \in \mathcal{C}^2_\sigma(\A, \comps)^\ast$ is a normal, virtual diagonal for $(\A,\A_\ast)$, its restriction to $\Csw((\A \tensor^\gamma \A)^\ast)$ is a $\Csw$-virtual diagonal. By Theorem \ref{CswDthm}, this means that $(\A,\A_\ast)$ is Connes-amenable.
\end{proof}
\begin{remark}
In \cite{Effros88}, it was shown that a von Neumann algebra is Connes-amenable if \emph{and only if} it has a normal, virtual diagonal. The same is true for the measure algebras of locally compact groups (\cite{Runde03a} and \cite{Runde03b}). However, there are Connes-amenable, dual Banach algebras that fail to have a normal, virtual diagonal (\cite{Runde06}; see also \cite{Runde15}).
\end{remark}
\par 
Following again \cite{EffrosRuan00}, we write $\tensor^\lambda$ for the (completed) injective tensor product of Banach spaces. Given a Banach algebra $\A$, a left Banach $\A$-module $E$, and a right Banach $\A$-module $F$, we turn $E^\ast \tensor^\lambda F^\ast$ into a Banach $\A$-bimodule via
\[
  a \cdot (\psi \tensor \phi) := \psi \tensor a \cdot \phi \quad\text{and}\quad (\psi \tensor \phi) \cdot a := \psi \cdot a \tensor \phi \qquad (a \in \A, \, \phi \in E^\ast, \, \psi \in F^\ast).
\]
For this $\A$-bimodule action, the canonical isometric embedding of $E^\ast \tensor^\lambda F^\ast$ into $(E \tensor^\gamma F)^\ast$ becomes an $\A$-bimodule homomorphism; we can thus say that $E^\ast \tensor^\lambda F^\ast$ ``is'' a closed submodule of $(E \tensor^\gamma F)^\ast$.
\par 
We shall require the following lemma in the next section:
\begin{lemma} \label{Cswlem}
Let $(\A,\A_\ast)$ be a dual Banach algebra, let $E$ be a left Banach $\A$-module, and let $F$ be a right Banach $\A$-module. Suppose that $X$ is a closed right submodule of $E^\ast$ contained in $\RCsw(E^\ast)$ and that $Y$ is a closed left submodule of $F^\ast$ contained in $\LCsw(F^\ast)$. Then $X \tensor^\lambda Y$ is a closed sub-bimodule of $(E \tensor^\gamma F)^\ast$ contained in $\Csw((E \tensor^\gamma F)^\ast)$.
\end{lemma}
\begin{proof}
It is straightforward from the definition of the bimodule action of $\A$ on $E^\ast \tensor^\lambda F^\ast$ that $X \tensor^\lambda Y$ is a closed submodule of $(E \tensor^\gamma F)^\ast$.
\par 
To see that $X \tensor^\lambda Y \subset \Csw((E \tensor^\gamma F)^\ast)$, fix $\phi \in X$, and define
\[
  \rho_\phi \!: F^\ast \to X \tensor^\lambda F^\ast, \quad \psi \mapsto \phi \tensor \psi.
\]
Then $\rho_\phi$ is a homomorphism of left Banach $\A$-modules, so that 
\[
  \rho_\phi(Y) \subset \rho_\phi(\LCsw(F^\ast)) \subset \LCsw(X \tensor^\lambda F^\ast) \subset \LCsw((E \tensor^\gamma F)^\ast)
\]
and, in particular, $\{ \phi \tensor \psi : \psi \in Y \} \subset \LCsw((E \tensor^\gamma F)^\ast)$. As $\phi \in X$ was arbitrary, it follows by linearity and continuity that $X \tensor^\lambda Y \subset \LCsw((E \tensor^\gamma F)^\ast)$.
\par
Analogously, we obtain that $X \tensor^\lambda Y \subset \RCsw((E \tensor^\gamma F)^\ast)$, so that $X \tensor^\lambda Y \subset \Csw((E \tensor^\gamma F)^\ast)$ as claimed.
\end{proof}
\par 
We note the following immediate consequence of Lemma \ref{Cswlem} (which is what we will actually need below):
\begin{corollary} \label{Cswcor}
Let $(\A, \A_\ast)$ be a dual Banach algebra. Then $\A_\ast \tensor^\lambda \A_\ast$ is a closed submodule of $(\A \tensor^\gamma \A)^\ast$ contained in $\Csw((\A \tensor^\gamma \A)^\ast)$. 
\end{corollary}
\section{Connes-amenability of $B(G)$}
We now turn to characterizing those locally compact groups $G$ for which the Fourier--Stieltjes algebra $B(G)$ is Connes amenable. We call $G$ \emph{almost abelian} if it has an abelian subgroup of finite index (other terms to describe this class of groups are ``virtually abelian'' or ``finite-by-abelian''). We shall prove that $B(G)$ is Connes-amenable if and only if $G$ is almost abelian. The method of proof resembles in some ways the one used in \cite{ForrestRunde05} and \cite{Runde08}: from the existence of a $\Csw$-virtual diagonal for $B(G)$, we conclude that a certain map is completely bounded, which is possible only if $G$ is almost abelian. However, in comparison with \cite{ForrestRunde05} and \cite{Runde08}, there are considerable technical difficulties to overcome. 
\par
Let $G$ be a locally compact group, and let $\tensor_{\max}$ stand for the maximal tensor product of $\cstar$-algebras (see \cite{Takesaki02}). Then $\cstar(G) \tensor_{\max} \cstar(G)$ is the enveloping $\cstar$-algebra of the Banach $^\ast$-algebra $L^1(G) \tensor^\gamma L^1(G)$. As $L^1(G) \tensor^\gamma L^1(G) \cong L^1(G \times G)$ isometrically isomorphic as Banach $^\ast$-algebras, this means that $\cstar(G) \tensor_{\max} \cstar(G) \cong \cstar(G \times G)$ canonically as $\cstar$-algebras. Since the identity on the algebraic tensor product $\cstar(G) \tensor \cstar(G)$ extends to a contraction from $\cstar(G) \tensor_{\max} \cstar(G)$ into $\cstar(G) \tensor^\lambda \cstar(G)$, we have---by virtue of Corollary \ref{Cswcor}---a canonical map $\theta \!: \cstar(G \tensor G) \to \Csw((B(G) \tensor^\gamma B(G))^\ast)$. Both $\cstar(G \times G)$ and $\Csw((B(G) \tensor^\gamma B(G))^\ast)$ are Banach $B(G)$-bimodules in a canonical way, and it is clear by construction that $\theta$ is a $B(G)$-bimodule homomorphism.
\par
We summarize:
\begin{lemma} \label{BGlem1}
Let $G$ be a locally compact group. Then there is a canonical contractive $B(G)$-bimodule homomorphism $\theta \!: \cstar(G \times G) \to \Csw((B(G) \tensor^\gamma B(G))^\ast)$.
\end{lemma}
\par
Let $G$ be a locally compact group, and let $\wstar(G) := C^\ast(G)^{\ast\ast}$. There is a canonical weak$^\ast$ continuous unitary representation $\omega \!: G \to \wstar(G)$, the \emph{universal representation} of $G$,
with the following universal property: for any (WOT continuous) unitary representation $\pi$ of $G$ on a Hilbert space, there is unique normal $^\ast$-homomorphism $\rho \!: \wstar(G) \to \pi(G)''$ such that $\pi = \rho \circ \omega$. Let $G_d$ stand for $G$ equipped with the discrete topology, and let $\omega_d$ denote the universal representation of $G_d$. Then, by the foregoing, there is a unique normal $^\ast$-homomorphism $\rho \!: W^\ast(G_d) \to W^\ast(G)$ such that $\omega = \rho \circ \omega_d$. (The pre-adjoint of $\rho$ is the canonical inclusion $\iota \!: B(G) \to B(G_d)$; see \cite{Eymard64}.)
\par
We now turn to the situation of Lemma \ref{BGlem1}. Taking the second adjoint of $\theta$, we obtain a weak$^\ast$-weak$^\ast$ continuous $B(G)$-bimodule homomorphism $\theta^{\ast\ast} \!: W^\ast(G \times G) \to  \Csw((B(G) \tensor^\gamma B(G))^\ast)^{\ast\ast}$. We note that that $(G \times G)_d = G_d \times G_d$; for simplicity, we write $\omega$ and $\omega_d$ for the universal representations of $G \times G$ and $G_d \times G_d$, respectively. By the foregoing, there is a unique normal $^\ast$-homomorphism $\rho \!: W^\ast(G_d \times G_d) \to W^\ast(G \times G)$ such that $\omega = \rho \circ \omega_d$, which is immediately seen to be a $B(G)$-bimodule homomorphism. Consequently, $\rho \circ \theta^{\ast\ast} \!: W^\ast(G_d \times G_d) \to \Csw((B(G) \tensor^\gamma B(G))^\ast)^{\ast\ast}$ is a weak$^\ast$-weak$^\ast$ continuous $B(G)$-bimodule homomorphism, the pre-adjoint $\Theta \!: \Csw((B(G) \tensor^\gamma B(G))^\ast)^\ast \to B(G_d \times G_d)$ of which is also a $B(G)$-bimodule homomorphism.
\par
All in all, we have:
\begin{lemma} \label{BGlem2}
Let $G$ be a locally compact group. Then there is a canonical contractive $B(G)$-bimodule homomorphism $\Theta \!: \Csw((B(G) \tensor^\gamma B(G))^\ast)^\ast \to B(G_d \times G_d)$ such that
\begin{equation} \label{eq1}
  \langle \boldsymbol{X}, \theta^{\ast\ast}(\omega(x,y)) \rangle = \Theta(\boldsymbol{X})(x,y) \qquad (\boldsymbol{X} \in \Csw((B(G) \tensor^\gamma B(G))^\ast)^\ast, \, x,y \in G).
\end{equation}
\end{lemma}
\par 
The equality (\ref{eq1}) in Lemma \ref{BGlem2} is clear from the construction of $\Theta$.
\begin{proposition} \label{BGprop}
Let $G$ be a locally compact group, let $\boldsymbol{D} \in \Csw((B(G) \tensor^\gamma B(G))^\ast)^\ast$ be a $\Csw$-virtual diagonal for $B(G)$, and let $\Theta \!: \Csw((B(G) \tensor^\gamma B(G))^\ast)^\ast \to B(G_d \times G_d)$ be as in Lemma \emph{\ref{BGlem3}}. Then $\Theta(\boldsymbol{D}) \in B(G_d \times G_d)$ is the indicator function of the diagonal subgroup of $G \times G$, i.e.,
\[
  G_\Delta := \{ (x,x) : x \in G \}.
\]
\end{proposition}
\begin{proof}
Consider the diagram
\begin{equation} \label{eq2}
  \begin{diagram}[midshaft]
  \Csw((B(G) \tensor^\gamma B(G))^\ast)^\ast & \rTo^\Theta & B(G_d \times G_d) \\
  \dTo^{\Delta_\sigma^w}                     &             & \dTo_{f \mapsto f|_{G_\Delta}} \\
  B(G)                                       & \rTo^\iota  & B(G_d) .
  \end{diagram}
\end{equation}
From (\ref{eq1}), we infer that (\ref{eq2}) commutes. As $\Delta_\sigma^w \boldsymbol{D} = 1$, we conclude that $\Theta(\boldsymbol{D}) |_{G_\Delta} = 1$.
\par
Let $(x,y) \in ( G \times G ) \setminus G_\Delta$, i.e., $x \neq y$. Choose $f,g \in B(G)$ such that $f(x) = 1 = g(y)$ and $fg = 0$. Then we have
\[
  \Theta(\boldsymbol{D})(x,y) = (f \cdot \Theta(\boldsymbol{D}) \cdot g)(x,y) = \Theta(f \cdot \boldsymbol{D} \cdot g)(x,y)= \Theta(\boldsymbol{D} \cdot (fg))(x,y) = 0.
\]
It follows that $\Theta(\boldsymbol{D})$ vanishes off $G_\Delta$. This completes the proof.
\end{proof}
\par 
Let $X$ be a set, and let $\alpha \!: X \to X$ be a map. Define
\[
  \alpha_\ast \!: \comps^X \to \comps^X, \quad f \mapsto f \circ \alpha.
\]
We define:
\begin{definition} \label{admdef}
Let $G$ be a locally compact group, and let $\alpha \!: G \to G$ be a map. We say that $\alpha$ is \emph{$\Csw$-admissible} if:
\begin{alphitems}
\item $\alpha_\ast(B(G)) \subset B(G)$;
\item $(\id \tensor \alpha_\ast)^\ast(\Csw((B(G) \tensor^\gamma B(G))^\ast)) \subset \Csw((B(G) \tensor^\gamma B(G))^\ast)$.
\end{alphitems}
\end{definition}
\begin{remark}
The requirement that $\alpha_\ast(B(G)) \subset B(G)$ forces $\alpha$ to be continuous. Also, $\alpha_\ast$ is an algebra homomorphism and thus automatically continuous by the classical Gelfand--Rickart theorem. Hence, (b) makes sense.
\end{remark}
\begin{example}
Let $G$ be a locally compact group. For $f \in B(G)$, define $\check{f} \in B(G)$ through $\check{f}(x) := f(x^{-1})$ for $x \in G$. Define
\[
  \mbox{\ }^\vee \!: B(G) \to B(G), \quad f \mapsto \check{f}.
\]
This means that $\mbox{\ }^\vee = \alpha_\ast$ where $\alpha(x) = x^{-1}$ for $x \in G$. It is well known that $\mbox{\ }^\vee$ is well defined, i.e., leaves $B(G)$ invariant, and is an isometry (\cite{Eymard64}). We claim that $\mbox{\ }^\vee$ is weak$^\ast$-weak$^\ast$ continuous. To see this, let $B$ denote the closed unit ball of $B(G)$ and note that $\mbox{\ }^\vee$ leaves $B$ invariant. As $L^1(G)$ is norm dense in $\cstar(G)$, the weak$^\ast$ topologies $\sigma(B(G),\cstar(G))$ and $\sigma(L^\infty(G),L^1(G))$ conincide on $B$. As $\mbox{\ }^\vee$ is evidently $\sigma(L^\infty(G),L^1(G))$-$\sigma(L^\infty(G),L^1(G))$ continuous on $B(G)$, we obtain that $\mbox{\ }^\vee$ is $\sigma(B(G),\cstar(G))$-$\sigma(B(G),\cstar(G))$ continuous on $B$, i.e., for each $x \in \cstar(G)$, the functional $B(G) \ni f \mapsto \langle x, \check{f} \rangle$ is $\sigma(B(G),\cstar(G))$ continuous on $B$. As consequence of the Krein-\v{S}mulian theorem (\cite[2.5.11.\ Corollary]{Pedersen88}), each such functional is continuous on \emph{all} of $B(G)$, which means that $\mbox{\ }^\vee$ is weak$^\ast$-weak$^\ast$ continuous as claimed.
\par 
Let $\Phi \in \Csw((B(G) \tensor^\gamma B(G))^\ast$, $f \in B(G)$, and $\boldsymbol{x} \in B(G) \tensor^\gamma B(G)$. We then have
\begin{equation} \label{eq3}
  \langle \boldsymbol{x}, (\id \tensor \mbox{\ }^\vee)^\ast(\Phi) \cdot f \rangle = \langle (\id \tensor \mbox{\ }^\vee)(f\cdot \boldsymbol{x} ), \Phi \rangle
  = \langle f \cdot (\id \tensor \mbox{\ }^\vee)(\boldsymbol{x}), \Phi \rangle = \langle \boldsymbol{x}, (\id \tensor \mbox{\ }^\vee)^\ast(\Phi \cdot f) \rangle
\end{equation}
and
\begin{equation} \label{eq4}
  \langle \boldsymbol{x}, f \cdot (\id \tensor \mbox{\ }^\vee)^\ast(\Phi) \rangle = \langle (\id \tensor \mbox{\ }^\vee)(\boldsymbol{x} \cdot f), \Phi \rangle
  = \langle (\id \tensor \mbox{\ }^\vee)(\boldsymbol{x}) \cdot \check{f}, \Phi \rangle = \langle \boldsymbol{x}, (\id \tensor \mbox{\ }^\vee)^\ast(\check{f} \cdot \Phi) \rangle.
\end{equation}
From (\ref{eq3}), it is clear that 
\[
  (\id \tensor \mbox{\ }^\vee)^\ast(\RCsw((B(G) \tensor^\gamma B(G))^\ast)) \subset \RCsw((B(G) \tensor^\gamma B(G))^\ast),
\]
and since $\mbox{\ }^\vee$ is weak$^\ast$-weak$^\ast$ continuous on $B(G)$, it follows from (\ref{eq4}) that 
\[
  (\id \tensor \mbox{\ }^\vee)^\ast(\LCsw((B(G) \tensor^\gamma B(G))^\ast)) \subset \LCsw((B(G) \tensor^\gamma B(G))^\ast)
\]
as well. All in all, $(\id \tensor \mbox{\ }^\vee)^\ast(\Csw((B(G) \tensor^\gamma B(G))^\ast)) \subset \Csw((B(G) \tensor^\gamma B(G))^\ast)$ holds, so that $G \ni x \mapsto x^{-1}$ is $\Csw$-admissible.
\end{example}
\begin{lemma} \label{BGlem3}
Let $G$ be a locally compact group, let $\boldsymbol{D} \in \Csw((B(G) \tensor^\gamma B(G))^\ast)^\ast$ be a $\Csw$-virtual diagonal for $B(G)$, let $\Theta \!: \Csw((B(G) \tensor^\gamma B(G))^\ast)^\ast \to B(G_d \times G_d)$ be as in Lemma \emph{\ref{BGlem2}}, and let $\alpha \!: G \to G$ be $\Csw$-admissible. Then 
\[
  \Theta(((\id \tensor \alpha_\ast)^\ast|_{\Csw((B(G) \tensor^\gamma B(G))^\ast)} )^\ast)(\boldsymbol{D}) \in B(G_d \times G_d)
\]
is the indicator function of
\[
  \{ (x,y) : x,y \in G, \, \alpha(y) = x \}.
\]
\end{lemma}
\begin{proof}
Let $x,y \in G$. By (\ref{eq1}), we have
\[
  \begin{split}
  \Theta(\boldsymbol{D})(x,\alpha(y)) & = \langle \boldsymbol{D}, \theta^{\ast\ast}(\omega(x,\alpha(y)) \rangle \rangle \\
  & = \langle \boldsymbol{D}, (\id \tensor \alpha_\ast)^{\ast\ast}(\theta^{\ast\ast}(\omega(x,y)) \rangle \\
  & = \langle (\id \tensor \alpha_\ast)^\ast(\boldsymbol{D}), (\theta^{\ast\ast}(\omega(x,y)) \rangle \\
  & = \Theta(((\id \tensor \alpha_\ast)^\ast|_{\Csw((B(G) \tensor^\gamma B(G))^\ast)} )^\ast)(\boldsymbol{D})(x,y) .
  \end{split}
\]
In view of Proposition \ref{BGprop}, this yields the claim.
\end{proof}
\par 
The following corollary of Lemma \ref{BGlem3} is crucial for the proof of our main result (Theorem \ref{BGthm} below). For the theory of operator spaces and, in particular, for the notion of a completely bounded map, we refer to the monograph \cite{EffrosRuan00}.
\begin{corollary} \label{BGcor1}
Let $G$ be a locally compact group, let $\boldsymbol{D} \in \Csw((B(G) \tensor^\gamma B(G))^\ast)^\ast$ be a $\Csw$-virtual diagonal for $B(G)$, let $\Theta \!: \Csw((B(G) \tensor^\gamma B(G))^\ast)^\ast \to B(G_d \times G_d)$ be as in Lemma \emph{\ref{BGlem2}}, and let $\alpha \!: G \to G$ be $\Csw$-admissible and bijective. Then $(\alpha^{-1})_\ast(B(G_d)) \subset B(G_d)$, and $(\alpha^{-1})_\ast \!: B(G_d) \to B(G_d)$ is completely bounded.
\end{corollary}
\begin{proof}
As $\alpha$ is bijective, we have that
\[
  \{ (x,y) : x, y \in G, x = \alpha(y) \} = \{ (x,\alpha^{-1}(x)) : x \in G \},
\]
i.e., is the graph of $\alpha^{-1}$, which we denote by $\mathfrak{Gr}(\alpha^{-1})$. By Lemma \ref{BGlem3}, the indicator function of $\mathfrak{Gr}(\alpha^{-1})$ lies in $B(G_d \times G_d)$, which, by Host's idempotent theorem (\cite{Host84}) entails that $\mathfrak{Gr}(\alpha^{-1})$ lies in the \emph{coset ring} of $G_d \times G_d$, i.e., the ring of subsets of $G_d \times G_d$ generated by all cosets of subgroups of $G_d \times G_d$.
This, in turn, implies (\cite[Lemma 1.2]{IlieSpronk05}) that $\alpha^{-1} \!: G_d \to G_d$ is a so-called \emph{piecewise affine map} (see \cite{IlieSpronk05} for the definition). Finally, we conclude from \cite[Corollary 3.2]{IlieSpronk05} that $(\alpha^{-1})_\ast(B(G_d)) \subset B(G_d)$, and that $(\alpha^{-1})_\ast \!: B(G_d) \to B(G_d)$ is completely bounded.
\end{proof}
\par 
We can now prove the central result of this paper:
\begin{theorem} \label{BGthm}
Let $G$ be a locally compact group. Then the following are equivalent:
\begin{items}
\item $B(G)$ is Connes-amenable;
\item $B(G)$ has a normal, virtual diagonal;
\item $G$ is almost abelian.
\end{items}
\end{theorem}
\begin{proof}
(iii) $\Longrightarrow$ (ii) is \cite[Proposition 3.1]{Runde04}, and (ii) $\Longrightarrow$ (i) is clear by general dual Banach algebra theory (Corollary \ref{VDcor}).
\par
(i) $\Longrightarrow$ (ii): Suppose that $B(G)$ is Connes-amenable and thus has a $\Csw$-virtual diagonal by Theorem \ref{CswDthm}. As $G \ni x \mapsto x^{-1}$ is $\Csw$-admissible, bijective, and its own inverse, we conclude from Corollary \ref{BGcor1} that $\mbox{\ }^\vee \!: B(G_d) \to B(G_d)$ is completely bounded. As $\mbox{\ }^\vee$ fixes the Fourier algebra $A(G)$ (see \cite{Eymard64} for its definition), and since the inclusion $A(G) \subset B(G_d)$ is a complete isometry, it follows that $\mbox{\ }^\vee \!: A(G) \to A(G)$ is completely bounded as well. By \cite[Proposition 1.5]{ForrestRunde05}, this is possible only if $G$ is almost abelian.
\end{proof}
\par 
If we replace $B(G)$ in Theorem \ref{BGthm} with $B_r(G)$, the corresponding equivalences remain true:
\begin{corollary} \label{BGcor2}
Let $G$ be a locally compact group. Then the following are equivalent:
\begin{items}
\item $B_r(G)$ is Connes-amenable;
\item $B_r(G)$ has a normal, virtual diagonal;
\item $G$ is almost abelian.
\end{items}
\end{corollary}
\begin{proof}
(iii) $\Longrightarrow$ (ii): If $G$ has an abelian subgroup of finite index, then $G$ is, in particular, amenable, so that $B_r(G) = B(G)$. Hence, $B_r(G)$ has a normal, virtual diagonal by \cite[Proposition 3.1]{Runde04}.
\par 
(ii) $\Longrightarrow$ (i) follows, as the corresponding implication in Theorem \ref{BGthm}, from general dual Banach algebra theory.
\par 
(i) $\Longrightarrow$ (iii): If $B_r(G)$ is Connes-amenable, it must have an identity (\cite[Proposition 4.1]{Runde01}). As $B_r(G)$ is a closed ideal of $B(G)$ (\cite[(2.16) Proposition]{Eymard64}, this means that $B_r(G) = B(G)$. The claim then follows from Theorem \ref{BGthm}.
\end{proof}
\begin{remark}
Both Theorem \ref{BGthm} and Corollary \ref{BGcor2} show that Connes-amenability is a rather restrictive condition to impose on Fourier--Stieltjes algebras. Comparing the main results of \cite{Ruan95} and \cite{ForrestRunde05}, one is led to believe that probably the best notion of amenability to apply to Fourier--Stieltjes algebras should be some hybrid notion of Connes-amenability and \emph{operator amenability}, as introduced by Z.-J.\ Ruan in \cite{Ruan95}. In \cite{RundeSpronk04}, the first-named author and N.\ Spronk introduced such a hybrid notion---appropriately dubbed \emph{operator Connes-amenability}. In view of \cite{Johnson72a}, \cite{Ruan95}, and \cite{Runde03a}, the natural guess was that, for a locally compact group $G$, both $B(G)$ and $B_r(G)$ are operator Connes-amenable if and only if $G$ is amenable. This is indeed true for $B_r(G)$; however, $B(\free_2)$, with $\free_2$ being the free group in two generators, is operator Connes-amenable even though $\free_2$ fails to be amenable (see \cite{RundeSpronk04}).
\end{remark}
\dated
\vfill
\begin{tabbing} 
\textit{Second author's address}: \= \kill
\textit{First author'sddress}:    \> Department of Mathematical and Statistical Sciences \\
                                  \> University of Alberta \\
                                  \> Edmonton, Alberta \\
                                  \> Canada T6G 2G1 \\[\medskipamount]
\textit{E-mail}:                  \> \texttt{vrunde@ualberta.ca} \\[\bigskipamount]
\textit{Second author's address}: \> Department of Mathematics and Statistics \\
                                  \> College of Arts and Sciences \\
                                  \> American University of Sharjah \\
                                  \> P.O.\ Box 26666 \\
                                  \> Sharjah \\
                                  \> United Arab Emirates \\[\medskipamount]
\textit{E-mail}:                  \> \texttt{fuygul@aus.edu}            
\end{tabbing}
\end{document}